[1994/06/01]% LaTeX date must June 1994 or later

\documentclass[12pt,reqno,intlimits,english]{amsart}

\usepackage[T1]{fontenc}

\usepackage{ifthen}
\usepackage{babel}
\usepackage{xspace}
\newcommand{\daTheoremname}{Theorem\xspace}
\newcommand{\daLemmaname}{Lemma\xspace}
\newcommand{\daPropositionname}{Proposition\xspace}
\newcommand{\daCorollaryname}{Corollary\xspace}
\newcommand{\daRemarkname}{Remark\xspace}
\newcommand{\daDefinitionname}{Definition\xspace}

\numberwithin{equation}{section}
\theoremstyle{plain}
\newtheorem{theorem}{\daTheoremname{}}[section]
\newtheorem{corollary}[theorem]{\daCorollaryname{}}
\theoremstyle{remark}
\newtheorem{Remark}[theorem]{\daRemarkname{}}
\newenvironment{remark}{\begin{Remark}}{\qed\end{Remark}}
\theoremstyle{definition}
\newtheorem{Definition}[theorem]{\daDefinitionname{}}
\newenvironment{definition}{\begin{Definition}}{\qed\end{Definition}}

\makeatletter
\newcommand{\textupmd}[1]{\textup{\textmd{#1}}}
\newcommand{\SpDim}{N}
\newcommand{\numberspacefont}{\boldsymbol}
\newcommand{\R}{\numberspacefont{R}}
\newcommand{\RN}{\R^{\SpDim}}
\newcommand{\abs}[1]{\lvert#1\rvert}
\newcommand{\Abs}[1]{\left|#1\right|}
\newcommand{\@di}{\textupmd{d}}
\newcommand{\di}{\,\@di}
\newcommand{\grad}{\operatorname{\nabla}}
\newcommand{\bdr}[1]{\partial #1}
\newcommand{\der}[3][1]{
\ifthenelse{ \equal{#1}{1} }{ \def\@deraux{} }{ \def\@deraux{#1} }
\frac{\@di^{\@deraux}#2}
            {\@di #3^{\@deraux}}}
\newcommand{\eps}{\varepsilon}

\makeatother
\begin{document}

\title
{Some remarks on the Sobolev inequality in Riemannian
manifolds}%
\author{Daniele Andreucci}
\address{Department of Basic and Applied Sciences for Engineering\\Sapienza University of Rome\\via A. Scarpa 16 00161 Rome, Italy}
\email{daniele.andreucci@sbai.uniroma1.it}
\thanks{The first author is member of the Gruppo Nazionale
  per la Fisica Matematica (GNFM) of the Istituto Nazionale di Alta Matematica
  (INdAM)}
\author{Anatoli F. Tedeev}
\address{South Mathematical Institute of VSC RAS\\Vladikavkaz, Russian Federation}
\email{a\_tedeev@yahoo.com}
\thanks{
\\
The first author is member of Italian
G.N.F.M.-I.N.d.A.M.}

\begin{abstract}
  We investigate Sobolev and Hardy inequalities, specifically weighted
  Minerbe's type estimates, in noncompact complete connected
  Riemannian manifolds whose geometry is described by an isoperimetric
  profile. In particular, we assume that the manifold satisfies the
  $p$-hyperbolicity property, stated in terms of a necessary integral Dini
  condition on the isoperimetric profile.  Our method seems to us to
  combine sharply the knowledge of the isoperimetric profile and the
  optimal Bliss type Hardy inequality depending on the geometry of the
  manifold.  We recover the well known best Sobolev constant in the Euclidean
  case.
\end{abstract}

\date{\today}

\maketitle

\section{Introduction}\label{s:intro}
The well known Sobolev inequality in the Euclidean space $\RN$ reads
\begin{equation}
  \label{eq:sobec}
  \Big(
  \int_{\RN}
  \abs{u}^{p^{*}}
  \di x
  \Big)^{1/p^{*}}
  \le
  S(N,p)
  \Big(
  \int_{\RN}
  \abs{\grad u}^{p}
  \di x
  \Big)^{1/p}
  \,,
\end{equation}
where $p^{*}=Np/(N-p)$, for all $u\in C^{\infty}_{0}(\RN)$; we always assume here $1<p<N$. The best constant $S$ was found in \cite{Aubin:1976,Talenti:1976b} as
\begin{equation}
  \label{eq:sobec_bc}
  S(N,p)
  =
  \frac{1}{\omega_{N}^{1/N}N}
  \Big[
  \frac{N(p-1)}{N-p}
  \Big]^{(p-1)/p}
  \,
  \Big[
  \frac{
    \Gamma(N+1)
  }{
    N\Gamma(N/p)\Gamma(1+N-N/p)
  }
  \Big]^{1/N}
  \,,
\end{equation}
where $\Gamma$ is the standard Gamma function and $\omega_{N}$ denotes the volume of the unit ball in $\RN$.

The literature on the subject of Sobolev inequalities and the choice of constants therein is very large; we refer to \cite{AubinNPRG, Mazja:ss}.

Let us make clear what the difficulty is in the case we have in mind,
that is the lack of homogeneity, which is instead guaranteed in the
Euclidean case. Consider a product manifold given by
$M=M_{0}\times \R^{k}$, where $M_{0}$ is a compact Riemannian manifold
of dimension $m$ and $\R^{k}$ the Euclidean space of dimension
$k$. Clearly its topological or local dimension is $N=m+k$, but its
dimension at infinity is lower, and in fact equals $k$. By this we
mean, more exactly, that if $\Omega\subset M$ is a smooth set with
volume $\abs{\Omega}_{N}=v$, its boundary has area satisfying
$\abs{\bdr{\Omega}}_{N-1}\ge c v^{(N-1)/N}$ if $v$ is small, but
$\abs{\bdr{\Omega}}_{N-1}\ge c v^{(k-1)/k}$ if $v$ is large (for a
suitable $c>0$).  The dimension at infinity $k$ is related to the
range of $p$ for which a Sobolev-like inequality is valid; this
amounts essentially to $p<k$ in simple cases and is strictly connected
to the property of $p$-hyperbolicity of the manifold; see
\eqref{eq:hyper} below.

Here we introduce a streamlined method of proof of Sobolev-like inequalities in Riemannian manifolds which seems to tackle optimally this setting, in terms of the isoperimetric information just exemplified (Theorem~\ref{t:sobol}); it yields the constant $S(N,p)$ in the Euclidean case; see Subsection~\ref{s:examp} for examples of other manifolds where the needed information is completely available. The inequality is of the type obtained in \cite{Minerbe:2009}. 

The connections between the validity of Sobolev-like inequalities and isoperimetric profiles (defined as the optimal $h$ in \eqref{eq:isoper_a} below) is well known; let us briefly recall that it appeared in \cite{Federer:Fleming:1960, Mazja:ss}. The method of \cite{Mazja:ss} allows one to reduce the proof of multidimensional Sobolev inequalities to one-dimensional Hardy type inequalities, and was also applied to derive Hardy inequalities in Riemannian manifolds by \cite{Miklyukov:Vuorinen:1999}. The symmetrization approach has also been used extensively in this field; we refer for example to \cite{Aubin:1976, Cianchi:2007, Martin:Milman:2014, Talenti:1976b, Xia:2001}; the optimality of the constant in Sobolev-like inequalities has been analyzed also, with alternative approaches, in \cite{Berchio:Ganguly:Grillo:2017, Cordero:Nazaret:Villani:2004, Hebey:Vaugon:1995, Ledoux:1999, Nazarov:2011}.

In what follows $(M,g)$ is a complete, connected Riemannian $N$-di\-men\-sional
manifold with infinite volume, $\di\mu$ is the volume form associated to the
metric $g$ , $\grad u$ denotes the gradient of a function $u$ with respect to
the metric $g$. Denote by $d(x)$ for $x\in M$ the distance from a fixed point
$x_{0}\in M$, and by $V(R)$ the volume of the geodesic ball $B_{R}(x_{0})$, $R>0$.

\begin{definition}
  \label{d:isoper}%1.2
  We say that $M$ satisfies the $h$-isoperimetric inequality if:
  \\
  i) for
  any measurable subset $U\subset M$ with Lipschitz continuous boundary $\partial U$
  \begin{equation}
    \label{eq:isoper_a}%1.1
    \abs{\partial U}_{N-1}
    \ge
    h(\mu(U))
    \,,
  \end{equation}
  where $h(s)$ is a given increasing function for $s\ge0$, $h(0)=0$;
  \\
  ii) the function $w$ defined by
  \begin{equation}
    \label{eq:isoper_b}%1.2
    w(s)
    =
    \frac{C_{N}s^{\frac{N-1}{N}}}{h(s)}
    \,,
    \qquad
    s>0
    \,,
  \end{equation}
  is nondecreasing, where $C_{N}>0$ is a given arbitrary constant.
\end{definition}
The constant $C_{N}$ in \eqref{eq:isoper_b} is essentially introduced for the sake of comparison with examples. Roughly speaking it is selected so that $h(s)\sim C_{N} s^{(N-1)/N}$ for small $s$; in the Euclidean case $C_{N}=N\omega_{N}^{1/N}$.

One of the main technical difficulties in investigating sharp Sobolev
constants in Riemannian manifolds is the already remarked
inhomogeneous character of the isoperimetric function $h$, which makes
difficult the explicit determination of isoperimetric regions (see \cite{Burago:Zalgaller:1988,Ros:2005} for classical references).
A property which certainly is necessary to us is $p$-hyperbolicity. This
essentially amounts to the existence of a symmetric positive Green
function $G_{x}$ for the $p$-Laplacian with pole at $x$, for every
$x\in M$.  For other  definitions of $p$-hyperbolicity and comments on its necessity we
refer to \cite{DAmbrosio:Dipierro:2014} (see also \cite{Miklyukov:Vuorinen:1999}). Here we need
\begin{equation}
  \label{eq:hyper}%1.3
  \int_{1}^{\infty}
  \frac{\di t}{h(t)^{p/(p-1)}}
  <
  +\infty
  \,,
\end{equation}
which is a geometrical version of the $p$-hyperbolicity assumption (see \cite{Grigoryan:1999a,Troyanov:1999}).

Next, we state some assumptions connected with
the validity of Sobolev and Hardy inequalities, i.e.,
\begin{equation}
  \label{eq:B1}%1.4
  B_{1}
  =
  \sup_{s>0}
  \left(
    \int_{0}^{s}
    \left[  
      \frac{h(\tau)}{C_{N}\tau^{\frac{N-1}{N}}}
    \right]^{p^{*}}
    \di\tau
  \right)^{1/p^{*}}
  \left(
    \int_{s}^{\infty}
    \frac{1}{h(\tau)^{\frac{p}{p-1}}}
    \di\tau
  \right)^{(p-1)/p}
  <
  +\infty
  \,,
\end{equation}
and
\begin{equation}
  \label{eq:B2}%1.5
  B_{2}
  =
  \sup_{s>0}
  \left(
    \int_{0}^{s}
    \frac{1}{\left(  V^{(-1)}(\tau)\right)^{p}}
    \di\tau
  \right)^{1/p}
  \left(
    \int_{s}^{\infty}
    \frac{1}{h(\tau)^{\frac{p}{p-1}}}
    \di\tau
  \right)^{(p-1)/p}
  <
  +\infty
  \,.
\end{equation}
Define also $k_{p,p}=p/(p-1)^{(p-1)/p}$, and
\begin{equation}
  \label{eq:kqp}%1.6
  k_{q,p}
  =
  \left(
    \frac{r}{\textup{Beta}(1/r,(q-1)/r)}
  \right)^{1/p-1/q}
  \,,
  \qquad
  \text{for $q>p$,}
\end{equation}
where $r=q/p-1$; here the Beta function is defined by
\begin{equation*}
  \textup{Beta}(a,b)
  =
  \int_{0}^{1}
  x^{a-1}(1-x)^{b-1}
  \di x
  \,.
\end{equation*}
Since $\textup{Beta}(a,b)=\Gamma(a)\Gamma(b)/\Gamma(a+b)$, $a\Gamma(a)=\Gamma(a+1)$, we have
\begin{equation*}
  k_{p^{*},p}
  =
  \frac{1}{N^{1/N}}
  \left\{
    \frac{\Gamma(N+1)}{\Gamma(N/p)\Gamma(1+N-N/p)}
  \right\}^{1/N}
  \,.
\end{equation*}

Next we state our main result, in the spirit of \cite{Minerbe:2009}; however our set of assumptions is different from the one there.

\begin{theorem}
  \label{t:sobol}%1.1
  Assume that $M$ satisfies the $h$-isoperimetric inequality and \eqref{eq:hyper} holds true; let
  $u\in C_{0}^{\infty}(M)$.
  \\
  i) If \eqref{eq:B1} holds, then
  \begin{equation}
    \label{eq:sobol_n}%1.7
    \left(
      \int_{M}
      \abs{ u}^{p^{*}}
      w(V(d(x))^{-p^{*}}
      \di\mu
    \right)^{1/p^{*}}
    \le
    C_{1}
    \left(
      \int_{M}
      \abs{ \grad u}^{p}
      \di\mu
    \right)^{1/p}
    \,.
\end{equation}
ii) If \eqref{eq:B2} holds, then
\begin{equation}
  \label{eq:sobol_nn}%1.8
  \int_{M}
  \frac{\abs{ u}^{p}}{d(x)^{p}}
  \di\mu
  \le
  C_{2}
  \int_{M}
  \abs{\grad u}^{p}
  \di\mu
  \,,
\end{equation}
Here $C_{1}=B_{1}k_{p^{*},p}$, and $C_{2}=B_{2}^{p}p^{p}/(p-1)^{p-1}$.
\end{theorem}

\begin{corollary}[Hardy-Sobolev inequality]
  \label{co:hsobol}
  Assume that $M$ satisfies the $h$-isoperimetric inequality and
  \eqref{eq:hyper}--\eqref{eq:B2} hold. Let
  \begin{equation*}
    q<p
    \,,
    \quad
    p^{*}(q):=\frac{N-q}{N-p}p
    \,.
  \end{equation*}
  Then
  \begin{multline}
    \label{eq:hsobol_n}%1.9
    \int_{M}
    \frac{\abs{ u}^{p^{*}(q)}}{d(x)^{q}}
    w(V(d(x)))^{-(p^{*}(q)-q)}
    \di\mu
    \\
    \le
    C_{1}^{N(p-q)/(N-p)}
    C_{2}^{q/p}
    \left(
      \int_{M}
      \abs{\grad u}^{p}
      \di\mu
    \right)^{(N-q)/(N-p)}
    \,.
\end{multline}
\end{corollary}

\begin{remark}
  \label{r:sobol_eucl}%1.1
  The constant in \eqref{eq:sobol_n} is sharp, at least in $\RN$ with the Euclidean metric. In this case indeed $h(t)=C_{N}
  t^{(N-1)/N}$, where $C_{N}=N\omega_{N}^{1/N}$ is the constant selected in the definition of $w$ so that $w(s)=1$ for all $s$. Then
  \begin{equation*}
    C_{1}
    =
    C_{N}^{-1}
    \left[
      \frac{N(p-1)}{N-p}
    \right]^{(p-1)/p}
    k_{p^{*},p}
    =
    S(N,p)
    \,,
  \end{equation*}
  the best constant in the Sobolev inequality \eqref{eq:sobec}. A similar result holds for $C_{2}$, see Subsection~\ref{s:examp_pw}.
\end{remark}

The function $w$ is an important geometrical characteristic of the
manifold and has been employed in the authors' papers
\cite{Andreucci:Tedeev:2000, Andreucci:Tedeev:2015} when studying embedding theorems of
Gagliardo-Nirenberg type and in the qualitative analysis of solutions
to nonlinear parabolic equations in both Euclidean and Riemannian
setting. In the form of \eqref{eq:sobol_n}
the Sobolev inequality firstly was proven in \cite{Minerbe:2009}. A short proof
of \eqref{eq:sobol_n} under assumptions stronger than \eqref{eq:B1}
was given in \cite{Andreucci:Tedeev:2021}.
\\
See for example \cite{Petean:Ruiz:2011} for other reasons of interest of isoperimetric profiles.

The plan of the paper is the following: below we give some examples of manifolds where our results can be applied. In Section~\ref{s:hardy} we recall a known Hardy-like inequality, extracting from it the consequences that we need. In Section~\ref{s:p_sobol} we give the proofs of our results. 

\subsection{Examples}
\label{s:examp}

\subsubsection{The case of product-like isoperimetric profiles.}
\label{s:examp_pro}

We discuss here the case when one can assume
\begin{equation}
  \label{eq:examp_pro_h}
  h(s)
  =
  \min(
  a
  s^{\lambda}
  ,
  b
  s^{\mu}
  )
  \,,
  \qquad
  s>0
  \,,
\end{equation}
where $\lambda>\mu>0$ and $a$, $b>0$. This is for example the case of subsets of $\RN$ shaped like paraboloids
\begin{equation*}
  \Omega_{\beta}
  =
  \{
  (x',x_{N})\in\RN
  \mid
  \abs{x'}<x_{N}^{\beta}
  \}
  \,,
\end{equation*}
where $\beta\in(0,1)$; in this case one has $\lambda=(N-1)/N$ and $\mu=\beta(N-1)/(1+\beta(N-1))$. However, we pursue here a different class of examples, that is the one arising from product manifolds in which a factor is an Euclidean space and the other one is compact. The problem of determining the isoperimetric profile even in this specific class of Riemannian manifolds is difficult (see \cite{Ruiz:Vazquez:2020} and references therein). In our examples, if $N$ is the topological dimension of the product manifold and $k<N$ its dimension at infinity, one has  \eqref{eq:examp_pro_h} with $\lambda=(N-1)/N$ and $\mu=(k-1)/k$; a lengthy but elementary explicit computation yields in this case that the $\sup$ defining $B_{1}$ in \eqref{eq:B1} is attained as $s\to+\infty$ and
\begin{equation}
  \label{eq:examp_pro_B1}
  B_{1}
  =
  C_{N}^{-1}
  \left[
    \frac{N-p}{N}
  \right]^{(N-p)/(Np)}
  \,
  \left[
    \frac{k}{k-p}
  \right]^{(N-1)/N}
  \,
  (p-1)^{(p-1)/p}
  \,.
\end{equation}
However, note that in the embedding \eqref{eq:sobol_n} the constants $a$ and $b$ enter the estimate through $w$; see \eqref{eq:examp_minerbe} below.
\\
In \cite{Pedrosa:Ritore:1999} it is considered the case of the product $S^{1}_{r}\times\R^{k}$, where $S^{1}_{r}$ is the circle of radius $r>0$; here $N=k+1$ and $2\le k\le 7$; the isoperimetric profile, that is the best possible $h$ in \eqref{eq:isoper_a}, is determined exactly in the form \eqref{eq:examp_pro_h}. In \cite{Ruiz:Vazquez:2020} the authors investigate quantitatively the case of $T^{m}\times\R^{k}$, $2\le k\le 7-m$, where $T^{m}$ is the $m$-dimensional flat torus; here of course $N=m+k$; in the case $m=2$ they obtain the exact profile \eqref{eq:examp_pro_h} for, respectively, small enough and large enough $s$; however the profile is quantitatively estimated for all $s$.
\\
In \cite{Ritore:Vernadakis:2017} it is considered the more general product manifold given by $M_{m}\times \R^{k}$, where $M_{m}$ is an $m$-dimensional compact Riemannian manifold without boundary and $k\ge 1$; let here $N=m+k$. It is proved that
the isoperimetric profile $h$ is given as in \eqref{eq:examp_pro_h}, but only the  constant $b$ is determined exactly, as
\begin{equation}
  \label{eq:examp_pro_b}
  b
  =
  k(\omega_{k}H^{m}(M_{m}))^{1/k}
  \,,
  %\text{.}\tag{2}%
\end{equation}
where $H^{m}$ denotes the $m$-dimensional Hausdorff measure.
\\
In all the cases just discussed, we may select $C_{N}=a$ in the definition of $w$, so that the constant $B_{1}$ is given by \eqref{eq:examp_pro_B1} and the (Minerbe type)
inequality \eqref{eq:sobol_n} reads
\begin{multline}
  \label{eq:examp_minerbe}
  \left(
    \int_{M_{m}\times\R^{k}}
    \abs{ u}^{p^{*}}
    \min\left\{  1,\frac{b}{a}V(d(x))^{-\frac{m}{k(m+k)}}\right\}^{p^{*}}
    \di\mu
  \right)^{1/p^{*}}
  \\
  \le
  C_{1}\left(
    \int_{M_{m}\times\R^{k}}
    \abs{ \grad u}^{p}
    \di\mu
  \right)^{1/p}
  \,,
  \quad
  p<k
  \,.
  %%%%%\tag{5}%
\end{multline}
The last assumption $p<k$ is needed to guarantee the $p$-hyperbolicity \eqref{eq:hyper}.

\subsubsection{Manifolds with bounded geometry.}
\label{s:examp_bdd}
Passing to a more general setting, following \cite[p.~136-137]{ChavelII}, let $M_{N}$ be an $N$-dimensional complete
Riemannian manifold with bounded geometry, that is such that its Ricci curvature is bounded below by a negative constant, and its injectivity radius is bounded below by a positive constant. For any $N\ge \nu>1$, $\rho>0$ define
the isoperimetric function
\begin{equation*}
  J_{\nu,\rho}(M_{N})
  :=
  \inf_{\Omega}
  \frac{\abs{\partial\Omega}_{N-1}}{V(\Omega)^{1-1/\nu}}
  \,,
\end{equation*}
where $\Omega$ varies over the open submanifolds of $M_{N}$ with
compact closure, $C^{\infty}$ boundary and containing a closed metric disk of radius
$\rho$. Then $J_{\nu,\rho}(M_{N})>0$ if and
only if there exists $v_{0}>0$ and $\theta>0$ such that
\begin{equation}
  \label{eq:examp_bdd_i}
  \abs{\partial\Omega}_{N-1}
  \ge
  \theta
  \begin{cases}
    V(\Omega)^{1-1/N}
    \,,
    &\qquad
    V(\Omega)\le v_{0}
    \,,
    \\
    V(\Omega)^{1-1/\nu}
    \,,
    &\qquad
    V(\Omega)\ge v_{0}
    \,.
  \end{cases}
\end{equation}
Select next $C_{N}=\theta$ in the definition of $w$. Then in \eqref{eq:examp_a} we can
choose $\alpha=(\nu-1)/\nu$ and we have from \eqref{eq:examp_c1}
\begin{equation*}
  C_{1}
  \le
  \theta^{-1}
  \left[
    \frac{N-p}{N}
  \right]^{1/p^{*}}
    \frac{(p-1)^{(p-1)/p}}{(\alpha p-p+1)^{(N-1)/N}}
  k_{p^{*},p}
  \,,
\end{equation*}
where $p^{\ast}=Np/(N-p)$.

\subsubsection{The case of power-like $h$.}
\label{s:examp_pw}
Let us consider a case where the finiteness of $B_{1}$ in \eqref{eq:B1} can be proved easily, but still with a majorization which is sharp in the Euclidean case. That is we assume that
\begin{equation}
  \label{eq:examp_a}%1.21
  \frac{t^{\alpha}}{h(t)}
  \quad
  \text{is nonincreasing for $t>0$,}
\end{equation}
where $(p-1)/p<\alpha\le(N-1)/N$; for example \eqref{eq:examp_a} holds for $h(s)=\gamma s^{k}[\ln(e+s)]^{z}$ for suitable $k$, $z$ for large $s$. Then we may estimate $B_{1}$ and
$B_{2}$ from above as follows. 

According to definition \eqref{eq:B1}, with an obvious definition of $J_{1}(s)$, $J_{2}(s)$,
\begin{equation*}
  B_{1}
  =
  \sup_{s>0}
  J_{1}(s)^{1/p^{\ast}}
  J_{2}(s)^{(p-1)/p}
  \,.
\end{equation*}
Let us estimate $J_{1}$ and $J_{2}$; by means of the change of
variable $\tau=sy$, using also \eqref{eq:examp_a}, we have
\begin{multline*}
  J_{1}
  =
  C_{N}^{-p^{*}}
  s
  \int_{0}^{1}
  \left(
    \frac{(  sy)^{\alpha} h(sy)}{(  sy)^{\alpha}(  sy)^{(N-1)/N}}
  \right)^{p^{*}}
  \di y
  \le
  \\
  C_{N}^{-p^{*}}
  s
  \left(
    \frac{h(s)}{s^{(N-1)/N}}
  \right)^{p^{*}}
  \int_{0}^{1}
  y^{-\left(  \frac{N-1}{N}-\alpha\right)p^{*}}
  \di y
  =
  \frac{C_{N}^{-p^{*}}
    \left(  N-p\right)  }{N(\alpha p-p+1)}
  s
  \left(
    \frac{h(s)}{s^{(N-1)/N}}
  \right)^{p^{*}}
  \,.
\end{multline*}
Analogously,
\begin{multline*}
  J_{2}=s
  \int_{1}^{\infty}
  \frac{(  sy)^{\alpha p/(p-1)}}{(  sy)^{\alpha p/(p-1)}h(sy)^{p/(p-1)}}
  \di y
  \le
  s
  \frac{1}{h(s)^{p/(p-1)}}
  \int_{1}^{\infty}
  \frac{\di y}{y^{\alpha p/(p-1)}}
  =
  \\
  \frac{p-1}{(\alpha p-p+1)}
  \frac{s}{h(s)^{p/(p-1)}}
  \,.
\end{multline*}
Finally, we have
\begin{equation}
  \label{eq:examp_c1}%1.25
  B_{1}
  \le
  C_{N}^{-1}
  \left[
    \frac{N-p}{N}
  \right]^{1/p^{*}}
  \frac{(p-1)^{(p-1)/p}}{(\alpha p-p+1)^{(N-1)/N}}
  \,,
\end{equation}
provided \eqref{eq:examp_a} holds. In particular, if $\alpha=(N-1)/N$, then
\begin{equation*}
  B_{1}
  \le
  C_{N}^{-1}
  \left[
    \frac{N(p-1)}{N-p}
  \right]^{(p-1)/p}
  \,,
\end{equation*}
leading to an estimate for the constant $C_{1}$ in \eqref{eq:sobol_n} which is in fact is the well known best constant in Sobolev inequality.

Let us calculate next $B_{2}$ in the Euclidean case where $h(\tau)=C_{N}\tau^{(N-1)/N}$, and $V(\tau)=\omega_{N}\tau^{N}$; an elementary calculation of the two integrals in the definition \eqref{eq:B2} of $B_{2}$ gives in this case $B_{2}=(p-1)^{(p-1)/p}/(N-p)$.
Therefore, in \eqref{eq:sobol_nn} we get $C_{2}=[p/(N-p)]^{p}$ which is the well known sharp constant in
the Euclidean Hardy inequality.

Let us continue with a more general case. Let us assume, besides \eqref{eq:examp_a}, that
\begin{equation}
  \label{eq:examp_b}%1.27
  h(s)
  \ge
  c_{0}
  \frac{s}{V^{(-1)}(s)}
  \,,
  \quad
  \text{for all $s>A$,}
\end{equation}
for some given constants $A$, $c_{0}>0$.
Since the metric is locally (i.e., for small $\tau$) Euclidean, we have that for a suitable constant $c_{1}$
\begin{equation}
  \label{eq:examp_i}
  \frac{1}{V^{(-1)}(\tau)}
  \le
  c_{1}
  \frac{h(\tau)}{\tau}
  \,,
  \qquad
  0<\tau\le A
  \,.
\end{equation}
For $\tau>A$ from \eqref{eq:examp_b} we have
\begin{equation}
  \label{eq:examp_ii}
  \frac{1}{  V^{(-1)}(\tau)^{p}}
  \le
  c_{0}^{-p}
  \frac{h(\tau)^{p}}{\tau^{p}}
  \,.
\end{equation}
Therefore integrating and exploiting as above \eqref{eq:examp_a}, we have for all $s>0$
\begin{equation*}
  \int_{0}^{s}
  \frac{1}{  V^{(-1)}(\tau)^{p}}
  \di\tau
  \le
  \max(c_{1}^{p},c_{0}^{-p})
  \int_{0}^{s}
  \frac{h(\tau)^{p}}{\tau^{p}}
  \di\tau
  \le
  \frac{\max(c_{1}^{p},c_{0}^{-p})}{\alpha p-p+1}
  \frac{h(s)^{p}}{s^{p-1}}
  \,,
\end{equation*}
where we have applied in the first inequality \eqref{eq:examp_i} and \eqref{eq:examp_ii}.
Then, on using again the estimate for the second factor below which we have proved above, we find
\begin{equation*}
  \left(
    \int_{0}^{s}
    \frac{1}{V^{(-1)}(\tau)^{p}}
    \di\tau
  \right)^{1/p}
  \left(
    \int_{s}^{\infty}
    \frac{1}{h(\tau)^{\frac{p}{p-1}}}
    \di\tau
  \right)^{p-1}
  \le
  \max(c_{1},c_{0}^{-1})
  \frac{(p-1)^{(p-1)/p}}{\alpha p-p+1}
  \,.
\end{equation*}
Thus the constant in \eqref{eq:sobol_nn} is estimated under the present assumptions by
\begin{equation}
  \label{eq:examp_j}%1.28
  C_{2}
  \le
  \left[
    \frac{p\max(c_{1},c_{0}^{-1})}{\alpha p-p+1}
  \right]^{p}
  \,.
\end{equation}

\section{The one-dimensional Hardy type inequality.}
\label{s:hardy}
The next statement, which
is a generalized Bliss \cite{Bliss:1930} inequality, is an important tool in our proof of
Theorem~\ref{t:sobol} and was proven in \cite{Li:Mao:2020} (see also \cite{Chen:2015})

\begin{theorem}[\cite{Li:Mao:2020}]
  \label{t:prel_lm}
  Let $1< p\le q <\infty$, $\mu$ and $\nu$
  be two $\sigma$-finite Borel measures on $\R$. Set
  \begin{equation*}
    B
    =
    \sup_{x\in \R}
    \nu((-\infty,x])^{(p-1)/p}
    \mu([x,+\infty))^{1/q}
    \,.
  \end{equation*}
  If $B<+\infty$ for all $f:\R\to \R$ we have
  \begin{equation}
    \label{eq:prel_lm_n}%2.8
    \left[
      \int_{\R}
      \Abs{
        \int_{-\infty}^{x}
        f(y)
        \di\nu_{y}
      }^{q}
      \di\mu_{x}
    \right]^{1/q}
    \le
    A
    \left[
      \int_{\R}
      \abs{f(x)}^{p}
      \di\nu_{x}
    \right]^{1/p}
    \,,
  \end{equation}
   for an optimal constant $A$ such that
\begin{equation}
  \label{eq:prel_lm_nn}%2.9
  B\le A\le k_{q,p}B
  \,,
\end{equation}
with $k_{q,p}$ defined in \eqref{eq:kqp}.
\end{theorem}
Note that \eqref{eq:prel_lm_n} can be seen as a generalized Hardy type inequality (see \cite{Mazja:ss,Opic:Kufner:hardy}).
It was shown in \cite{Chen:2015}, see also \cite{Li:Mao:2020}, that the estimate of $A$ in \eqref{eq:prel_lm_nn} is sharp.

We draw from Theorem~\ref{t:prel_lm} the following consequences, by means of standard changes of variables.

\begin{corollary}
  \label{co:prel_lmm}
  Let $1< p\le q <\infty$, $\mu$ and $\nu$
  be two $\sigma$-finite Borel measures on $[0,+\infty)$. Set
  \begin{equation}
    \label{eq:prel_lmm_nn}%2.11
    \widetilde B
    =
    \sup_{x\ge 0}
    \nu([x,\infty))^{(p-1)/p}
    \mu([0,x]))^{1/q}
    \,.
  \end{equation}
  If $\widetilde B<+\infty$ for all $f:[0,+\infty)\to \R$ we have
  \begin{equation}
    \label{eq:prel_lmm_n}%2.10
    \left[
      \int_{0}^{\infty}
      \Abs{
        \int_{x}^{\infty}
        f(y)
        \di\nu_{y}
      }^{q}
      \di\mu_{x}
    \right]^{1/q}
    \le
    A
    \left[
      \int_{0}^{\infty}
      \abs{ f(x)}^{p}
      \di\nu_{x}
    \right]^{1/p}
    \,,
  \end{equation}
  for an optimal constant $A$ such that $\widetilde B \le A\le k_{q,p}\widetilde B$.
  \\
  In particular, choosing
  \begin{equation*}
    q
    =
    p^{*}
    \,,
    \qquad
    \di\nu_{y}
    =
    \frac{dy}{h(y)^{p/(p-1)}}
    \qquad
    \di\mu_{x}
    =
    \left(
      \frac{h(x)}{C_{N}x^{(N-1)/N}}
    \right)^{p^{*}}
    \di x
    \,,
  \end{equation*}
  we obtain from \eqref{eq:prel_lmm_n} that
  \begin{multline}
    \label{eq:prel_lmm_m}%2.12
    \left[
      \int_{0}^{\infty}
      \Abs{
        \int_{x}^{\infty}
        f(y)
        \frac{\di y}{h(y)^{p/(p-1)}}
      }^{p^{*}}
      \left(
        \frac{h(x)}{C_{N}x^{(N-1)/N}}
      \right)^{p^{*}}
      \di x
    \right]^{1/p^{*}}
    \\
    \le
    A
    \left[
      \int_{0}^{\infty}
      \abs{ f(x)}^{p}
      \frac{\di x}{h(x)^{p/(p-1)}}
    \right]^{1/p}
    \,,
  \end{multline}
  provided \eqref{eq:B1} holds true and we select $\widetilde B=B_{1}$. Next, choosing
  \begin{equation*}
    q=p
    \,,
    \quad
    \di\nu_{y}
    =
    \frac{\di y}{h(y)^{p/(p-1)}}
    \,,
    \quad
    \di\mu_{x}
    =
    \frac{\di x}{(V^{(-1)}(x))^{p}}
    \,,
  \end{equation*}
  we have
  \begin{multline}
    \label{eq:prel_lmm_mm}%2.13
      \int_{0}^{\infty}
      \Abs{
        \int_{x}^{\infty}
        f(y)
        \frac{\di y}{h(y)^{p/(p-1)}}
      }^{p}
      \left(
        \frac{1}{V^{(-1)}(x)}
      \right)^{p}
      \di x
    \\
    \le
    A^{p}
      \int_{0}^{\infty}
      \abs{ f(x)}^{p}
      \frac{\di x}{h(x)^{p/(p-1)}}
    \,,
  \end{multline}
  provided \eqref{eq:B2} holds true and we select $\widetilde B=B_{2}$.
\end{corollary}

\section{Proof of Theorem~\ref{t:sobol} and of Corollary~\ref{co:hsobol}.}
\label{s:p_sobol}

\subsection{Preliminaries}\label{s:prel}
In what follows we will use some notions from geometrical measure theory.
Let $u$ be a measurable function defined on $M$. Denote
\begin{equation*}
  \nu(t):=
  \Abs{
    \{  x\in M \mid
    \abs{ u} >t\}
  }
  \,.
\end{equation*}
and let $u^{*}(s)$ be the decreasing rearrangement of $u$ defined on
$[ 0,\infty]$ as, roughly speaking, the generalized inverse to its distributional function, and more exactly as
\begin{equation*}
  u^{*}(0)=\sup\abs{ u}
  \,,
  \qquad
  u^{*}(s)=\inf\{ t\mid\nu(t)<s\}
  \,.
\end{equation*}
Let us recall some basic facts that we are going to use.

i) Cavalieri's principle, which is consequence of equimeasurability of sets $\{
\abs{u} >t\}$ and $\{u^{*}(s)>t\}$:
\begin{equation}
  \label{eq:prel_cav}%2.1
  \int_{M}
  \abs{u}^{p}
  \di\mu
  =
  \int_{0}^{\infty}
  u^{\ast}(s)^{p}
  \di s
  \,.
\end{equation}

ii) Hardy-Littlewood inequality
\begin{equation}
  \label{eq:prel_hl}%2.2
  \int_{M}
  uv
  \di\mu
  \le
  \int_{0}^{\infty}
  u^{*}(s)v^{*}(s)
  \di s
  \,.
\end{equation}

iii) Federer co-area formula (we refer the reader to \cite{Mazja:ss} in the Euclidean setting
and \cite{AubinNPRG} for manifolds). For any smooth enough functions $v$ and
$u$ defined on $M$ we have
\begin{equation}
  \label{eq:prel_fed}%2.3
  \int_{M}
  v\abs{\grad u}
  \di\mu
  =
  \int_{0}^{\infty}
  \di\tau
  \int_{\abs{ u} =\tau}
  v(x)
  \di s_{N-1}
  \,.
\end{equation}

iv) Polya-Szeg\"{o} inequality. Assume that $M$ satisfies the $h$
isoperimetric inequality \eqref{eq:isoper_a}. Then
\begin{equation}
  \label{eq:prel_ps}%2.4
  \int_{0}^{\infty}
  h(s)^{p}
  \left(  -\der{u^{*}}{s}(s)\right)^{p}
  \di s
  \le
  \int_{M}
  \abs{\grad u}^{p}
  \di\mu
  \,.
\end{equation}
We give a short proof of \eqref{eq:prel_ps}, for the readers's convenience.
First note that setting in \eqref{eq:prel_fed} $v=1$, we get
\begin{equation}
  \label{eq:prel_ps_i}%2.5
  P(t):=
  \Abs{
    \{  \abs{ u} >t\}
  }_{N-1}
  =
  -
  \der{}{t}
  \int_{\abs{ u} >t}
  \abs{ \grad u}
  \di\mu
  \,.
\end{equation}
Next, on applying H\"{o}lder inequality we obtain
\begin{multline*}
  \frac{1}{\eps}
  \int_{t<\abs{u}\le t+\eps}
  \abs{\grad u}
  \di\mu
  \le
  \left(
    \frac{1}{\eps}
    \int_{t<\abs{ u} \le t+\eps}
    \abs{\grad u}^{p}
    \di\mu
  \right)^{1/p}
  \\
  \left(
    \frac{1}{\eps}
    \Abs{
      \{  t<\abs{ u} \le t+\eps\}
    }
  \right)^{(p-1)/p}
  \,.
\end{multline*}
Letting $\eps\to 0$ in this inequality and noting that for $q\ge 1$
\begin{equation*}
  \lim_{\eps\to 0}
  \frac{1}{\eps}
  \int_{t<\abs{u} \le t+\eps}
  \abs{\grad u}^{q}
  \di\mu
  =
  -
  \der{}{t}
  \int_{t<\abs{ u} }
  \abs{\grad u}^{q}
  \di\mu
  \,,
\end{equation*}
and by using also \eqref{eq:prel_ps_i} we have that
\begin{equation}
  \label{eq:prel_ps_ii}%2.6
  P(t)
  \le
  \left(
    -
    \der{}{t}
    \int_{t<\abs{ u}}
    \abs{\grad u}^{p}
    \di\mu
  \right)^{1/p}
  \left(
    -
    \der{}{t}
    \nu(t)
  \right)^{(p-1)/p}
  \,.
\end{equation}
Using now \eqref{eq:isoper_a}, we get from \eqref{eq:prel_ps_ii}
\begin{equation}
  \label{eq:prel_ps_iii}%2.7
  h(\nu(t))^{p}
  \left(
    -
    \der{}{t}
    \nu(t)
  \right)^{-(p-1)}
  \le
  -
  \der{}{t}
  \int_{t<\abs{ u} }
  \abs{\grad u}^{p}
  \di\mu
  \,.
\end{equation}
Set $\nu(t)=s$, then $t=u^{*}(s)$, $\nu_{t}(t)=(u_{s}^{*}(s))^{-1}$ a.e.,
and therefore from \eqref{eq:prel_ps_iii} we get
\begin{equation*}
  h(s)^{p}
  \left(
    -
    \der{}{s}
    u^{*}(s)
  \right)^{p}
  \le
  \der{}{s}
  \int_{\abs{u} >u^{*}(s)}
  \abs{\grad u}^{p}
  \di\mu
  \,.
\end{equation*}
Integrating the last inequality between $0$ and $\infty$, we arrive at \eqref{eq:prel_ps}.

\begin{proof}[Proof of Theorem~\ref{t:sobol}]
By definition of
decreasing rearrangement we have
\begin{equation*}
  \left(
    \int_{M}
    u^{*}(s)^{p^{*}}
    \big(
    w(V(d(x)))^{-p^{*}}
    \big)^{*}
    \di s
  \right)^{1/p^{*}}
  =
  \left(
    \int_{M}
    u^{*}(s)^{p^{*}}
    w(s)^{-p^{*}}
    \di s
  \right)^{1/p^{*}}
  \,.
\end{equation*}
Therefore, by the Hardy-Littlewood inequality we obtain
\begin{equation}
  \label{eq:p_sobol_i}%3.1
  \left(
    \int_{M}
    \abs{ u}^{p^{*}}
    w(V(d(x)))^{-p^{*}}
    \di\mu
  \right)^{1/p^{*}}
  \le
  \left(
    \int_{0}^{\infty}
    \big[  u^{*}(s)\big]^{p^{*}}
    w(s)^{-p^{*}}
    \di s
  \right)^{1/p^{*}}
  \,.
\end{equation}
Next, by the Polya-Szeg\"{o} principle
\begin{equation}
  \label{eq:p_sobol_ii}%3.2
  \left(
    \int_{M}
    \abs{ \grad u}^{p}
    \di\mu
  \right)^{1/p}
  \ge
  \left(
    \int_{0}^{\infty}
    \big[  -u_{s}^{*}(s)\big]^{p^{*}}
    h(s)^{p}
    \di s
  \right)^{1/p}
  \,.
\end{equation}
Combining now \eqref{eq:p_sobol_i} and \eqref{eq:p_sobol_ii}, we deduce
\begin{equation*}
  \frac{
    \left(
      \int_{M}
      \abs{ \grad u}^{p}
      \di\mu
    \right)^{1/p}
  }{
    \left(
      \int_{M}
      \abs{ u}^{p^{*}}
      w(V(d(x)))^{-p^{*}}
      \di\mu
    \right)^{1/p^{*}}}
  \ge
  \frac{
    \left(
      \int_{0}^{\infty}
      [  -u_{s}^{*}(s)]^{p}
      h(s)^{p}
      \di s
    \right)^{1/p}
  }{
    \left(
      \int_{0}^{\infty}
      [  u^{*}(s)]^{p^{*}}
      w(s)^{-p^{*}}
      \di s
    \right)^{1/p^{*}}
  }
  \,.
\end{equation*}
In order to apply \eqref{eq:prel_lmm_m} let
\begin{gather}
  \label{eq:p_sobol_k}
  q=p^{*}
  \,,
  \quad
  u^{*}(s)
  =
  \int_{s}^{+\infty}
  f(y)
  \frac{\di y}{h(y)^{p/(p-1)}}
  \,,
  \\
  \label{eq:p_sobol_kk}
  \text{i.e.,}
  \quad
  f(s)
  =
  -u_{s}^{*}(s)
  h(s)^{p/(p-1)}
  \,.
\end{gather}
Then \eqref{eq:prel_lmm_m} implies the inequality
\begin{equation*}
  \left(
    \int_{0}^{\infty}
    \big[  u^{*}(s)\big]^{p^{*}}
    w(s)^{-p^{*}}
    \di s
  \right)^{1/p^{*}}
  \le
  A
  \left(
    \int_{0}^{\infty}
    (-u_{s}^{*}(s))^{p}
    h(s)^{p}
    \di s
  \right)^{1/p}
  \,,
\end{equation*}
that is the desired result \eqref{eq:sobol_n}, when we replace the best constant $A$ with its sharp estimate as in Section~\ref{s:hardy}.
\\
Let us prove \eqref{eq:sobol_nn} proceeding in the same way.
Note that $(1/d(\cdot)^{p})^{*}(s)=(V^{(-1)}(s))^{-p}$; then
we have by
the Hardy-Littlewood inequality that
\begin{equation*}
  \int_{M}
  \frac{\abs{ u}^{p}}{d(x)^{p}}
  \di\mu
  \le
  \int_{0}^{\infty}
  \big(  u^{*}(\tau)\big)^{p}
  (1/d(\cdot)^{p})^{*}(\tau)
  \di\tau
  =
  \int_{0}^{\infty}
  \frac{\left(  u^{*}(\tau)\right)^{p}\di\tau}{\left(  V^{(-1)}(\tau)\right)^{p}}
  \,.
\end{equation*}
On selecting $f$ as in \eqref{eq:p_sobol_k}--\eqref{eq:p_sobol_kk},
we have from \eqref{eq:prel_lmm_mm} that
\begin{equation*}
  \int_{0}^{\infty}
  \frac{(  u^{*}(\tau))^{p}\di\tau}{(  V^{(-1)}(\tau))^{p}}
  \le
  A^{p}
  \int_{0}^{\infty}
  h(\tau)^{p}
  \left(  -u_{\tau}^{*}(\tau)\right)^{p}
  \di\tau
  \,.
\end{equation*}
Finally, by making use of the Polya-Szeg\"{o} inequality \eqref{eq:prel_ps} we arrive at \eqref{eq:sobol_nn}.
\end{proof}

\begin{proof}[Proof of Corollary~\ref{co:hsobol}]
  Indeed, on applying H\"{o}lder inequality (splitting the exponent of $\abs{u}$ as $q+p^{*}(q)-q$) we have
  \begin{equation*}
    \begin{split}
    &\int\limits_{M}
    \frac{\abs{ u}^{p^{*}(q)}}{d(x)^{q}}
    w(V(d(x)))^{-(p^{*}(q)-q)}
    \di \mu
    \\
    &\quad
    \le
    \left(
      \int\limits_{M}
      \frac{\abs{ u}^{p}}{d(x)^{p}}
      \di\mu
    \right)^{q/p}
    \left(
      \int\limits_{M}
      \abs{u}^{p^{*}}
      w(V(d(x)))^{-p^{*}}
      \di\mu
    \right)^{(p-q)/p}
    \\
    &\quad\le
    C_{1}^{N(p-q)/(N-p)}
    C_{2}^{q/p}\left(
      \int\limits_{M}
      \abs{\grad u}^{p}
      \di\mu
    \right)^{(N-q)/(N-p)}
    \,,
    \end{split}
  \end{equation*}
  where in last inequality we used \eqref{eq:sobol_n}, \eqref{eq:sobol_nn}.
\end{proof}

\def\cprime{$'$}

% \bibliographystyle{abbrv}
% \bibliography{paraboli,pubblicazioni_andreucci}
\end{document}